\documentclass[12pt]{amsart}
\usepackage{amssymb} 

\newtheorem{thm}{Theorem}[section]
\newtheorem{cor}[thm]{Corollary}
\newtheorem{lem}[thm]{Lemma}
\newtheorem{prop}[thm]{Proposition}
\newtheorem{exam}[thm]{Example}

\theoremstyle{definition}
\newtheorem{defn}[thm]{Definition}
\newtheorem{rem}[thm]{Remark}

\newtheorem{que}[thm]{Question}

\numberwithin{equation}{section}

\usepackage[colorlinks]{hyperref}

\begin{document}
\title{ The order of the product of two elements in the periodic
 groups }
\author{M. Amiri, I. Lima\\
}
\footnotetext{E-mail Address: {\tt mohsen@ufam.edu.br;\, igor.matematico@gmail.com } }
\date{}
\maketitle

\begin{quote}
{\small \hfill{\rule{13.3cm}{.1mm}\hskip2cm} \textbf{Abstract.}
   Let $G$ be a periodic  group, and let 
$LCM(G)$ be the set of all $x\in G$ such that
$o(x^nz)$ divides the least common   multiple of $o(x^n)$ and $o(z)$ for all $z$ in $G$ and  all integers $n$.  In this paper, we prove that  the subgroup generated by 
$LCM(G)$ is a  locally nilpotent characteristic subgroup of $G$ whenever $G$ is a locally finite group. For $x,y\in G$ the vertex  $x$ is connected to vertex $y$ whenever $o(xy)$ divides the least common   multiple of $o(x)$ and $o(y)$.  Let $Deg(G)$  be the sum of all $deg(g)$ where $g$ runs over $G$.   We prove that  for any finite group $G$ with $h(G)$ conjugacy classes,    $Deg(G)=|G|(h(G)+1)$ if and only if  $G$ is an abelian group.\\[.2cm]
}\\
\vspace{-3mm}\hfill{\rule{13.3cm}{.1mm}\hskip2cm}
\end{quote}
\section{Introduction}
One of the oldest problem in group theory is, given two elements $x,y$ in a group $G$, of orders
$m$ and $n$, respectively, to find information on the order of the product $xy$. 
Understanding even
the easier problem of when $ab$ has finite order would have great implications in group theory,
for instance in the study of finitely-generated groups in which the generators have finite order. Burnsid’s problem is a good
 example of a difficult problem which asks whether
a finitely generated periodic group is necessarily finite. A negative answer to this problem  has been
provided in 1964 by Golod and Shafarevich \cite{Gol}, \cite{Sha}, although many variants of this question
still remain unsolved to this day. For more information on this subject, see  Kostrikin \cite{Kos}, Novikov and Adian \cite{Nov}, Ivanov and Ol'shanski\u{i}
\cite{12}, \cite{13}, \cite{23}, Zelmanov \cite{26}, \cite{27} and Lys\"{e}nok \cite{21}.
 Let $G$ be a periodic group, and let  $H$ and $R$ be two non-empty subsets of $G$. Let
$LCM(H,R)$ be the set of all $x\in H$ such that  and 
$o(x^nz)$ divides the least common   multiple  $lcm(o(x^n),o(z))$ for all $z$ in $R$ and all integers $n$. In particular, we denote the set $LCM(G,G)$ by $LCM(G)$ and the subgroup $\langle LCM(G)\rangle$ by $LC(G)$. In this paper, we prove that  the subgroup generated by 
$LCM(G)$ is a    nilpotent characteristic subgroup of $G$   whenever $G$ is a   finite group.  As consequence  of this result, we   prove  the following theorem:

\begin{thm}   Let $G$  be a locally finite  periodic group.
Then 
 $  LC(G)$ is a locally nilpotent subgroup of $G$.
\end{thm}

Let $G$ be a  group.
Define $LC_1=LC$ and for $i=2,3,...$, define
$LC(\frac{G}{LC_{i-1}(G)})=\frac{LC_{i}(G)}{LC_{i-1}(G)}$.
We say the group $G$ is a $LCM$-nilpotent group, whenever 
there exists a finite $LCM$-series 
\begin{equation}\label{eq11}
LC_1(G)\leq LC_2(G)\leq...\leq LC_k(G)=G
\end{equation}
such that $\frac{LC_i(G)}{LC_{i-1}(G)}$ is a nilpotent group for all $i=2,3,...,k$. In this case the $LCM$-series (\ref{eq11}), is call a nilpotent $LCM$-series for $G$.
As application of this definition, we prove the following proposition:
\begin{prop}
Let $G$ be a finite  $LCM$-nilpotent group of class $t$, and let $H=P_1\times ....\times P_k$ where $P_i\in Syl_{p_i}(G)$ and $\pi(|G|)=\{p_1,...,p_k\}$.
If $LC(\frac{G}{LC_{i-1}(G)})=LCM(\frac{G}{LC_{i-1}(G)})$ for all $i=2,...,t$, then there exists a bijection $f$ from $G$ to $H$ such that
$o(x) \mid o(f(x))$ for all $x\in G$.
\end{prop}

Let $G$ be a periodic group.
The $LCM$-graph $ \Gamma(G)$ is an  undirected  graph constructed as follows:

(i) Each element $g$ of $G$ is assigned a vertex.

(ii) For $x,y\in G$, the vertex  $x$ is connected to vertex $y$ whenever $o(xy)\mid lcm(o(x),o(y))$.

The degree (or valency) of a vertex of a graph is the number of edges that are incident to the vertex, and in a multigraph, loops are counted twice.
Let $H$ be a subgrup of $G$ and $g\in G$. The degree of a vertex $v$ in $H$ is denoted $deg_H(v)$ and whenever $G=H$, for simplicity, we denote $deg_G(g)$ by  $deg(g)$.   We denote by $Deg(G)$ to be the number $\sum_{g\in G}deg(g)$.
In the last section we find some application for $LC(G)$ in this special graph. Precisely, we prove the following theorems:

\begin{thm}
Let $G$ be a group of order  $n=p_1p_2\ldots p_k$ where $p_1<\ldots <p_k$ are primes. 
Let $p_i$ be the smallest prime divisor of $n$ such that $p_i\mid p_j-1$ for some prime divisor $p_j$ of $n$.
Then  $$Deg(G)\leq Deg(C_{n/p_rp_i}\times (C_{p_r}\rtimes C_{p_i}))$$
where $p_r=min\{p_j\mid n:  p_i\mid p_j-1\}$.
\end{thm}

\begin{thm}
Let $G$ be a finite group with $h(G)$  conjugacy classes. Then  $Deg(G)=|G|(h(G)+1)$ if and only if  $G$ is an abelian group.
\end{thm}

Let $CP2$ be the class of finite groups $G$ such that $o(xy)\leq max\{o(x), o(y)\}$ for  all $ x,y \in G$ \cite{Deb}. For a $p-$group $G$ we denote $\Omega_{i}(G)=<\{x\in G| \ x^{p^{i}}=1\}>$ and $\mho_{i}(G)=<\{x^{p^i}\in G| \ x\in G\}>$ for all $i\in \mathbb{N}$, respectively. 
Note that the class $CP2$ of $p$-groups is  larger than the class of abelian $p$-groups, regular $p-$groups (see Theorem 3.14 of \cite{Su},II, page 47) and $p-$groups whose subgroup lattices are modular (see Lemma 2.3.5 of \cite{Sch}). Moreover, by the main theorem in \cite{W}, we infer that powerful $p-$groups for $p$ odd also belong to $CP2$.

In what follows, we adopt the notation established in the Isaacs' book on finite groups  \cite{I}.  
\section{LCM(G)}
  
We shall need the following  theorem about the groups belonging to $CP2$.

\begin{thm}(Theorem D in \cite{Deb}) \label{a} A finite group $G$ is contained in $CP2$ if and only if one of the following statements holds:

\begin{enumerate}
  \item $G$ is a $p-$group and $\Omega_{n}(G)=\{x\in G\ | \ x^{p^{n}}=1\}$ for all integers $n$.
  \item $G$ is a Frobenius group of order $p^{\alpha}q^{\beta}$, $p<q$, with kernel $F(G)$ of order $p^{\alpha}$ and cyclic complement.
\end{enumerate}

\end{thm}

\begin{defn}    Let $G$  be a periodic  group,  and let $H$ and $R$ be two non-empty   subsets of $G$. Let  $LCM(H,R)$  be the set of all $x\in H$ such that 
 $o(x^ny)\mid lcm(o(x^n),o(y))$ for all $y\in R$ and all integers $n$. 
In addition, if $LCM(G)=G$, then we say $G$ is an \it{$LCM$}\it{-group}. 
The subgroup generated by $LCM(G)$ is denoted by $LC(G)$.
\end{defn} 
\begin{exam}
Let $G=F\rtimes H$ be a Frobenius group. Let $P\in Syl_p(F)$ such that $P\in CP2$. Let $x\in P$, and let  $n$ be an integer.
For all $h\in G\setminus F$, we have $o(x^nh)=o(h)\mid lcm(o(x^n),o(h))$.
Since $F$ is a nilpotent  group and $P\in CP2$, for 
all $z\in F$, we have $o(x^nz)\mid lcm(o(x^n),o(z))$.
It follows that $P\subseteq LCM(G)$. In particular, if $F$ is an abelian group, then $LCM(G)=F$.
\end{exam}

\begin{lem} \label{12}   Let $G$  be a periodic  group. Let $H$ and $R$ be two
$G$-invariant subsets of $G$.
Then $LCM(H,R)$ is a $G$-invariant subset of $G$. Also, if $H^{-1}=H$ and $R^{-1}=R$, then for all $x\in LCM(H,R)$, we have $x^{-1}\in LCM(H,R)$.
\end{lem}
\begin{proof} {
Let $\sigma\in Aut(G)$, and let $v\in LCM(H,R)$, and let $x\in \langle v\rangle$.   For all $y\in R$, we have $o(xy)=o(\sigma(xy))=o(\sigma(x)\sigma(y))$. 
Since $\sigma(R)=R$, there exists $g\in R$ such that $\sigma(g)=z$. Clearly, $o(z)=o(\sigma(g))$. For all $z\in R$, we have $$o((\sigma(x)z)=o(\sigma(xg))=o(xg)\mid lcm(o(x),o(g))=lcm(o(\sigma(x)),o(z)).$$

So $\sigma(x)\in LCM(H,R)$. Also, for all $z\in R$, we have  

$$o(x^{-1}z)=o((z^{-1}x)^{-1})=o(z^{-1}x).$$ It follows that  $$o(x(z^{-1}x)x^{-1})=o(xz^{-1})\mid lcm(o(x),o(z^{-1}))=lcm(o(x^{-1}),o(z)).$$

Hence $x^{-1}\in LCM(H,R)$.
}

\end{proof}

 Minimal non-nilpotent groups are characterized by Schmidt as follows:

\begin{thm}{\rm (see (9.1.9) of \cite{Rob})} \label{11} Assume that every maximal subgroup of a finite group $G$
is nilpotent but $G$ itself is not nilpotent. Then:

(i) $G$ is solvable.

(ii) $|G| = p^mq^n $ where $p$ and $q$ are unequal prime numbers.

(iii) There is a unique Sylow $p$-subgroup $P$ and a Sylow $q$-subgroup $Q$ is cyclic.
Hence $G = QP$ and $P\trianglelefteq G $.

\end{thm}

In the following theorem we give   a necessary and sufficient condition for a  finite group  to be an  $LCM-$group. 
\begin{thm}\label{1}   Let $G$  be a finite  group. 
Then  $G$ is an $LCM$-group if and only if  $G$ is a nilpotent group and each Sylow subgroup of $G$ belongs to $CP2$.

\end{thm}
\begin{proof} {
If  $G$ is a nilpotent group and all Sylow subgroups of $G$ belong to $CP2$, then clearly, $LCM(G)=G$.
For the other side  suppose that $LCM(G)=G$.
 Suppose for a contradiction  that there exists a minimal non-nilpotent subgroup $A$   of $G$. From
Theorem \ref{11}, $A=S\rtimes \langle a \rangle$ where $gcd(|S|,|\langle a\rangle|)=1$. Let $s\in S$ such that $s\not\in N_G(\langle a\rangle)$.
 We have $o(a(a^{-1}s))=o(s)$. Then $o(s)\mid lcm(o(a),o(a^{-1}s))=o(a)$, which is a contradiction.
  Hence $G$ is a nilpotent group.  
Let $P\in Syl_p(G)$. Since for all $x,y\in P$, we have $o(xy)\mid lcm(o(x),o(y))=max\{o(x),o(y)\}$, we have $P\in CP2$ by definition of $CP2$.}
\end{proof}
 Let $G$ be a periodic group.
  Let $p$ be a positive integer, and let $R_p(G)=\{x\in G:  gcd(o(x),p)\neq 1\}$ and
$H_p(G)=\{x\in G: gcd(o(x),p)=1\}$. 
 
The following lemma is useful to prove Theorems \ref{2} and \ref{thm}.
\begin{lem}\label{ces}
Let $G$ be a periodic  group and $p$ a positive integer. Let $A$ be a subset of $G$ such that $H_p(G)\subseteq A$.

(i)  If  $x\in LCM(H_p(G),A)$, then $exp(\langle x^G\rangle)=o(x)$.

(ii) If $x\in LC(R_p(G), G)$, then $exp(\langle x^G\rangle)=o(x)$.

In particular, $LC(G)$ is a characteristic nilpotent subgroup of $G$.
\end{lem}
\begin{proof}{    Let $S$ be a subset of $x^G$  such that  $F:=\langle x^G\rangle=\langle S\rangle$ but $F\neq \langle T\rangle$ for all proper subsets $T$ of $S$.
Any $w\in F$ is just a finite sequence $w=s_{1}\ldots s_{r}$ whose entries $ s_{1},\ldots ,s_{r}$ are elements of $S\cup S^{-1}$. The integer $r$ is called the length of the element $w$ and  its  norm $|w|$ with respect to the generating set $S$ is defined to be the shortest length of  $w$ over $S$. 
 
 (i) Let $y\in F.$ 
We proceed by induction on $|y|$.   The case $|y|=0$ is trivial.
So suppose that  the result is true for all $a\in F$ with  $|a|<|y|$. There are $s_1,\ldots,s_k\in S$, $\epsilon_i\in\{1,-1\}$ and positive integers $n_1,\ldots,n_k$  such that
 $y=(s_1)^{\epsilon_1n_1}\ldots(s_k)^{\epsilon_kn_k}$. 
 We may assume that $n_1>0$.
By induction hypothesis, we have $o((s_1)^{\epsilon_1(n_1-1)}\ldots(s_k)^{\epsilon_kn_k})\mid o(x)$. It follows that 
$gcd( o((s_1)^{\epsilon_1(n_1-1)}\ldots(s_k)^{\epsilon_kn_k}),p)=1$.
Since $x^g,(x^g)^{-1}\in LCM(H_p(G),A)$ for all $g\in G$, we have 
$S\subseteq LCM(H_p(G), A)$. Consequently,

\begin{eqnarray*}
o(y)&=&o(s_1((s_1)^{\epsilon_1 (n_1-1)}(s_k)^{\epsilon_2 n_2}\ldots(s_k)^{\epsilon_k n_k})))\\&\mid& lcm(o(s_1), o((s_1)^{\epsilon_1 (n_1-1)}(s_2)^{\epsilon_2 n_2}\ldots(s_k)^{\epsilon_k n_k}))\mid\\&\vdots&\\&\mid&
lcm(o(s_1),o(x))\\&=&o(x).
\end{eqnarray*}

(ii)It is similar to the case (i) with a little change.

In particular, if $x\in LCM(G)$, then $x\in LCM(H_p(G),G)$ or
$x\in LCM(R_p,G)$.   Therefore the proof is clear from  parts (i) and (ii).
}

\end{proof}

Now, we are ready to prove one of our main results of this section. 
 
\begin{thm}\label{thm}    Let $G$  be a finite  group, and let $p$ be a positive integer. Then $$LCM(H_p(G), G) \cup  LCM(R_p(G), G)\subseteq Fit(G).$$

\end{thm}
\begin{proof}{First we prove that $LCM(H_p(G), G)\subseteq Fit(G))$.
Let $x\in LCM(H_p(G),G)$, and let $H=\langle x^G\rangle$.  From Lemma \ref{ces}, 
$exp(H)=o(x)$. If $o(x)$ is a power of prime number $p$, then 
clearly, $H$ is a $p$-group. So suppose that $o(x)=p_1^{\alpha_1}p_2^{\alpha_2}...p_k^{\alpha_k}$ where $p_1<p_2<...<p_k$ are prime numbers.
Then $x=v_1v_2...v_k$ where $o(v_i)=p_i^{\alpha_i}$ for all $i=1,2,...,k$ and $v_iv_j=v_jv_i$ for all $1\leq i\leq j\leq k$.
Let $H_i=\langle v_i^G\rangle$ for all $1\leq i\leq j\leq k$.
For the reason that each $H_i$  is  nilpotent normal $p_i$-subgroups of $G$, we have $H=H_1H_2...H_k$ is a nilpotent subgroup.
 Let $H_{x_i}=\langle x_i^G\rangle$ where $x_i\in LCM(H_p(G),G)=\{x_1,\ldots,x_k\}$.
 Then $C=H_{x_1}H_{x_2}\ldots H_{x_k}$ is a nilpotent subgroup of $G$.
 Since $LCM(H_p(G),G)\leq C$, we have $LCM(H_p(G),G)\subseteq Fit(G)$.
 
Similarly, we can prove that  $LCM(H_p(G), G)\subseteq Fit(G))$.

 }
\end{proof}

\begin{cor}\label{thm2}    Let $G$  be a finite  group. Then $LC(G)$ is a nilpotent characteristic subgroup of $G$. 
\end{cor}
\begin{proof}
Let $p$ be a prime number which is co-prime to $|G|$.
From Theorem \ref{thm}, $$\langle LCM(H_p(G),G)\rangle =\langle LCM(G,G)\rangle=LC(G)$$ is a nilpotent characteristic subgroup of $G$. 
\end{proof}

\begin{thm}\label{2}    Let $G$  be a locally finite  periodic group.
Then 
 $  LC(G)$ is a locally nilpotent subgroup of $G$.
\end{thm}
\begin{proof} 
{    Suppose for a contradiction that 
$LC(G)$ is not a locally nilpotent subgroup of $G$. Then there  exist $z_1, ...,z_t\in LCM(G)$ such that
$\langle z_1,...,z_t\rangle$ is not a nilpotent group. 
Let $H=\langle z_1,...,z_t\rangle$. Since $G$ is locally finite,  $H$ is a finite group.   From Corollary \ref{thm2},
 $\langle  z_1,...,z_t\rangle\leq LC(G)$ is a nilpotent,   which is a contradiction.
 }
 \end{proof}
 
 The following example shows that the condition $x^n\in LCM(G)$ for all integers $n$ in Corollary \ref{thm2}, is necessary.
\begin{exam}
Let $G=N_1\times ...\times N_{30}$  where $N_i\cong A_5$ for all $i=1,2,...,30$, and let 
$\sigma:=(1,2,...,30)$ be a cyclic permutation of order $30$ 
on components of $G$.
Let $H=G\rtimes \langle \sigma\rangle$.
Since  for all $h\in G$, we have $o(h\sigma)\mid 30=lcm(o(\sigma),o(h))$.
But $\langle \sigma^H\rangle=G\rtimes \langle \sigma\rangle$ is not a solvable group.

\end{exam}

In the following example we introduce  an infinite group $G$ such that its    $LC(G)$ subgroup is not a nilpotent subgroup.
\begin{exam}
The idea of the following example   works with the theory of hyperbolic groups (in the Gromov sense, \cite{Gro}) and Kazhdan's Property (T), which the reader can see the definitions in \cite{Osh} and \cite{Kaz}. Fix a prime $p>2$ and denote by $G$ the free product of two copies of $Z/pZ$. This group is hyperbolic (see Lemma 1.18 in \cite{Osh}). Via iterated small cancellation in hyperbolic groups (see section 4 in \cite{Osh}), we may construct periodic quotients of that group (see proof of Corollary 2 in \cite{Osh}), and this leads to the following. For $n=p^k$ with $k$ large enough, the quotient $F(p,k)$ of $G$ by the smallest normal subgroup of $G$ that contains all elements $g^n$, $g$ in $G$, is an infinite group. It is also a non-amenable group. To see this, we may argue as follows. The Ol'shanski\u{i} Common Quotient Theorem says that two hyperbolic groups $H$ and $H'$ always have a common quotient (see Theorem 5.7 and Corollary 5.8 in \cite{Cap}), and one can moreover assume that the torsion elements in the quotient come from the torsion elements of $H$, or of $H'$. In particular, there is a quotient $H$ of $G$ with Kazhdan's property (T) whose torsion elements have order $p$. Then, if $k$ is large enough, the quotient of $H$ by the smallest normal subgroup containing all elements $h^n$ with $h$ in $H$ is an infinite Kazhdan group; and this implies that $F(p,k)$ is not amenable (indeed since by compactness the only discrete groups that satisfy both property (T) and amenability are finite). By definition, $<LCM(F(p,k))>$ is $F(p,k)$ itself; in particular it is not nilpotent.

\end{exam}

\begin{cor} \label{4}   Let $G$  be a finite  group. 

(i) If $Fit(G)=1$, then $LC(G)=1$.

(ii) If   $N\lhd G$, then $LC(N)\leq Fit(G)$.
\end{cor}
\begin{proof}{
From Theorem \ref{thm},  $LC(G)$ is a normal nilpotent subgroup of $G$.

(i)
Since $Fit(G)=1$, we have $LC(G)=1$.
	
(ii)From Theorem \ref{thm},  $LC(N)$ is a nilpotent characteristic subgroup of $N$.
Hence $LC(N)\leq Fit(G)$.}
\end{proof}
\begin{cor}\label{hall}  Let $G$ be a finite solvable group.
If $Fit(G)\neq G$ is a Hall subgroup such that for each Sylow subgroup $P$ of $Fit(G)$, we have $P\in CP2$, then $Fit(G)=LC(G)$.

\end{cor}
\begin{proof}{

Let $y\in G\setminus Fit(G)$ and let $x\in Fit(G)$.
We have $(yx)^{m}=x^{y^{-1}}x^{y^{-2}}\ldots x^{y^{-m}}y^{m}$.
If $m=o(y)$, then 
$(yx)^{o(y)}=x^{y^{-1}}x^{y^{-2}}\ldots x^{y^{-m}}\in Fit(G)$.
Now, since every Sylow subgroup of $Fit(G)$ belongs to $CP2$,  we have
$(x^{y^{-1}}x^{y^{-2}}\ldots x^{y^{-m}})^{o(x)}=1$. It follows that 
$o(yx)=o(xy)\mid o(x)o(y)=lcm(o(x),o(y))$.

}
\end{proof}

\begin{cor}\label{ccc}
Let $G$ be a non-nilpotent finite group such that $2\nmid |Fit(G)|$.
Then the following statements are equivalent:

(i) 
All proper subgroups  of  $G$ are $LCM$-groups.

(ii) 
 $G$ is a minimal non-nilpotent group.

\end{cor}
\begin{proof}
{ (i)$\Rightarrow$ (ii):  From Theorem \ref{thm}, all subgroups of $G$ are nilpotent. So $G$ is a minimal non-nilpotent group.

(ii)$\Rightarrow$(i):  If $G$ is a minimal non-nilpotent group, then from Theorem \ref{11},
there is a unique Sylow $p$-subgroup $P$ and a cyclic Sylow $q$-subgroup $Q$  such that $G=P\rtimes Q$.
From Lemma 2.3 of  \cite{Jafar},
$\frac{G}{Z(G)}$ is Frobenius group such
that the Frobenius kernel is elementary abelian and the Frobenius complement is of prime
order. Hence the nilpotency classes of $P$ is at most two. Since $2\nmid |Fit(G)|$,  $p>2$, so   $P$ is a regular group, and then $P\in CP2$.
 Now, from Corollary \ref{hall},
all subgroups of $G$ are $LCM$-groups.

}
\end{proof}
Note that   Corollary  \ref{ccc}, in general, is not true. For example,    a GAP \cite{Gap} check yields, the SmallGroup(160,199) is a minimal non-nilpotent group such that its  Sylow $2$-subgroup is not  a  $LCM$-group.

\begin{thm}  Let $G$ be a finite  group and $H$ be  a locally solvable  periodic group.

(i) If $G\setminus LCM(G)\subseteq \{x\in G: o(x)=exp(G)\}$, then 
$G$ is a nilpotent group.

(ii) Suppose that $[x,y]\in LCM(G)$ for all $x,y\in G$. Then 
$o(uv)\mid lcm(o(u),o(v))\cdot o([u,v])$ for all $u,v\in G$.

(iii) Suppose that $[x,y]\in LCM(H)$ for all $x,y\in H$. Then 
$o(uv)\mid lcm(o(u),o(v))\cdot o([u,v])$ for all $u,v\in H$.

\end{thm}
\begin{proof}{
(i) Clearly, we may assume that $G$ is not a $p$-group. Then $exp(G)$ has at least two distinct prime divisors. 
Let $P\in Syl_p(G)$. By our assumption 
$P\subseteq LCM(G)$. From Theorem \ref{thm}, $LC(G)$ is normal and  nilpotent subgroup of $G$, so $P\lhd G$. It follows that every Sylow subgroup of $G$ is normal in $G$, so $G$ is a nilpotent group.

(ii) Let $u,v\in G$, and let $m=lcm(o(u),o(v))$.
From Hall-Petresco formula, we have $(uv)^m=u^mv^mc_2^{(^m_2)}\ldots c_{m-1}^{(^m_{m-1})}c_m=c_2^{(^m_2)}\ldots c_{m-1}^{(^m_{m-1})}c_m$ where  $c_i\in \gamma_i(\langle u, v\rangle)$. 
Since $\gamma_1(\langle u, v\rangle)=\langle \{[u,v]^g: g\in \langle u, v\rangle\}\rangle$ and $[u,v]\in LCM(G)$, from Lemma \ref{ces}, we have 
$exp(\gamma_1(\langle u, v\rangle))=o([u,v])$.
Then $(uv)^{m\cdot o([u,v])}=(c_2^{(^m_2)}\ldots c_{m-1}^{(^m_{m-1})}c_m)^{o([u,v])}=1$, as claimed.

(iii) It is similar to the case (ii).

}
\end{proof}

In the following example, we show that in general,  the equality  $Fit(G)=LC(G)$ is not true.
\begin{exam}
Let $G$ be a $p$-group of maximal class and order $p^n>p^{p+1}$ such that $exp(G)>p^3$. Let $M$ be the regular maximal subgroup of $G$.
Let $x\in M$ of order $exp(G)$ and $y\in G\setminus M$.
Then $o(xy^{-1})\leq p^2$ and $o(y)\leq p^2$. 
Since $o(xy^{-1}y)=exp(G)\nmid p^2=lcm(o(xy^{-1},o(y))$, we have 
$y\not\in LCM(G)$.
Consequently, $LCM(G)\subseteq M$, and so $LC(G)\leq M\neq G$.
\end{exam}
\section{$LCM$-series}
Let $G$ be a  group.
Define $LC_1=LC$ and for $i=2,3,...$, define
$LC(\frac{G}{LC_{i-1}(G)})=\frac{LC_{i}(G)}{LC_{i-1}(G)}$.
The series 

\begin{equation}\label{eq1}
LC_1(G)\leq LC_2(G)\leq...\leq LC_i(G)\leq ...
\end{equation}

is called $LCM$-series of $G$. We say the $LCM$-series (\ref{eq1}), is finite whenever there exists a positive integer $k$ such that $LC_{i+k}(G)=LC_k(G)$ for all integers $i=1,2,...$. The smallest integer $k$ such that 
$LC_{i+k}(G)=LC_k(G)$ for all integers $i=1,2,...$ is called length of series and $G$ is called $LCM$-nilpotent of class $k$.
\begin{defn}
We say the group $G$ is a $LCM$-nilpotent group, whenever 
there exists a finite $LCM$-series 
\begin{equation}\label{eq11}
LC_1(G)\leq LC_2(G)\leq...\leq LC_k(G)=G
\end{equation}
such that $\frac{LC_i(G)}{LC_{i-1}(G)}$ is a nilpotent group for all $i=2,3,...,k$. In this case the $LCM$-series (\ref{eq11}), is call a nilpotent $LCM$-series for $G$.
\end{defn}
\begin{exam}
Let $G$ be a non-abelian group of order $pq$ where $p<q$ are primes.
Then $G$ has a nilpotent $LCM$-series of length two.
Also, any nilpotent groups is a  $LCM$-nilpotent group.
Furthermore, $G$ is a finite $LCM$-nilpotent group of class one if and only if 
$G$ is a nilpotent group and all Sylow subgroups of $G$ are in $CP2$.
 
\end{exam}

For any positive integer  $n$, let $\pi(n)$ be the set of all prime divisors of $n$.
\begin{lem}
Let $G$ be a finite  $LCM$-nilpotent group of class $t$ and order $n=p_1^{\alpha_1}p_2^{\alpha_2}...p_k^{\alpha_k}$ where $p_1<...<p_k$ are prime numbers. Let $H=P_1\times...\times P_k$ where $P_i\in Syl_{p_i}(G)$  for all $i=1,...,k$. 
If $LC(\frac{G}{LC_{j-1}(G)})=LCM(\frac{G}{LC_{j-1}(G)})$ for all $j=2,...,t$, then there exists a bijection $f$ from $G$ to $H$ such that
$o(x) \mid o(f(x))$ for all $x\in G$.
\end{lem}
\begin{proof}
{
We proceed by induction on $|G|$. Let $N$ be a normal minimal subgroup of $G$ such that $N\leq LC_1(G)$. We may assume that $N\leq P_1$.
Let $Q_i=\frac{P_i}{P_i\cap N}$ for $i=1,2,...,k$.
By induction hypothesis, there exists a bijection $\theta$ from
$\frac{G}{N}$ onto $Q_1\times ...\times Q_k$ such that
$o(xN)\mid o(\theta(xN))$ for all $xN\in \frac{G}{N}$.
Let $\theta(N)=M$. Let $x\in G$. Since $G$ is a solvable group, we may assume that
$G=P_1P_2...P_k$. Then $x=x_1x_2...x_k$ where $x_i\in P_i$ for all $i=1,2,...,k$. Let $H=N\rtimes\langle x\rangle$. Set $o(x)=m$, and $gcd(m,|P_1|)=c$. 
We have $ H=(N\langle x^{m/c}\rangle) \rtimes \langle x^{c}\rangle$. 
  Let  $h\in N$. If $x_1h=1$, then 
  $o(xh)=o(x_1h,x_2,...,x_k)$. So suppose that $x_1h\neq 1$.
Clearly, $m/c\mid o(xh)$. 
Then $(xh)^{m/c}=h^{x^{-1}}h^{x^{-2}}...h^{x^{-m/c}}x^{m/c}\in N\langle x^{m/c}\rangle$. If $c>p$, then $N\leq LC(N\langle x^{m/c}\rangle)$, and so
$o(xh)=o(x)=o(x_1h,x_2,...,x_k)$.
So suppose that $c=p$. If $h^{x^{-1}}h^{x^{-2}}...h^{x^{-m/c}}=x^{-m/p}$, then
$h\in \langle x\rangle$, and so
$o(xh)=o(x)=o(x_1h,x_2,...,x_k)$.
 If $h^{x^{-1}}h^{x^{-2}}...h^{x^{-m/c}}\neq x^{-m/p}$, then
 $o(xh)=o(x)\mid o(x_1h,x_2,...,x_k)$.
 Hence we may define a bijection $f_x$ from $xN$ onto
 $\theta(x)M$ such that $o(xh)\mid o(f_x(xh))$ for all $h\in N$.
 Let $X$ be a left transversal for $N$ in $G$.
Then $f= \bigcup_{x\in X}f_x$ has the desired property.   
 }
\end{proof}

Given a finite group $G$, let $\psi(G)=\sum\limits_{x\in G}o(x)$. 
Many studies have been done on the function $\psi$,   to find an exact upper or lower  bound for sums
of element orders in non-cyclic finite groups.
Hence the following corollaries  could be useful to find some bound for the value of function $\psi$.

\begin{cor}
Let $G$ be a finite  $LCM$-nilpotent group of class $t$ and order $n=p_1^{\alpha_1}p_2^{\alpha_2}...p_k^{\alpha_k}$ where $p_1<...<p_k$ are prime numbers. Let $H=P_1\times...\times P_k$ where $P_i\in Syl_{p_i}(G)$  for all $i=1,...,k$. 
If $LC(\frac{G}{LC_{j-1}(G)})=LCM(\frac{G}{LC_{j-1}(G)})$ for all $j=2,...,t$, then $\psi(G)\leq \psi(H)$.
\end{cor}  We denote  by $\Delta(n)$ the set of all  finite    $LCM$-nilpotent groups $G$  of class $t_G$ such that  $LC(\frac{G}{LC_{i-1}(G)})=LCM(\frac{G}{LC_{i-1}(G)})$ for all $i=2,...,t_G$.
\begin{cor}
Let $G\in \Delta(n)$ such that $G$ is not a nilpotent group. Then there exists $H\in \Delta(n)$ such that $H$ is nilpotent and $\psi(G)\leq \psi(H)$.
\end{cor}

\begin{rem}
Let $G$ be a finite group. Let $x,y\in LCM(G)$ such that  $o(xy)=lcm(o(x),o(y))$. 
Let   $h\in G$. We have 
\begin{eqnarray*}
o(xyh)&\mid& lcm(o(x),o(yh))\\&\mid& lcm(o(x),o(y),o(h))\\&=&lcm(o(xy),o(h)).
\end{eqnarray*}
It follows that $xy\in LCM(G)$.
\end{rem}
With the computational group theory system GAP, the G:=SmallGroup(16,13) the $LC(G)$ is different from $LCM(G)$. In view of this example,
answering to the following questions would be interesting.

\begin{que}
(1) What is the set of all finite groups   $G$,  with $LC(G)=LCM(G)$?

(2) What is the set of all $LCM$-nilpotent groups of class two?
\end{que}

\section{LCM-Graph}

We start this section with the following  proposition  which  determines a normal Sylow subgroup in a finite group  using its graph.
\begin{prop}\label{gra}
Let $G$  be a finite group, and let $P\in Syl_p(G)$.
If all elements of $P$ are connected with all elements of 
$G\setminus P$, then $P\lhd G$.
\end{prop}
\begin{proof}
{  

Suppose there exists $y\in \Omega_{|P|}(G)\setminus P$. For all $x\in P$, 
$o(xy)$ is a power of $p$.  
It follows that $xy\in \Omega_{|P|}(G)$, and so $xy\in P^g$ for some $g\in G$. Clearly, all elements of $P^g$ are connected with all elements of 
$G\setminus P^g$.
Let $t\in P$. If $t\in P^g$, then $xyt\in P^g$, and so $o(xyt)$ is a power of $p$.
If $t\not\in P^g$, then $t\in G\setminus P^g$, and so 
$o(xyt)\mid lcm(o(xy),o(t))$. 
So $xyt\in \Omega_{|P|}(G)$.
Hence the group generated by $\langle P,y\rangle$ is a $p$-group, and so
$y\in P$, which is a contradiction. It follows that $P$ is the unique subgroup of $G$, and so
$P\lhd G$.
}
 \end{proof} 
 
Let $G$ and $H$ be two finite groups of the same order.
Let $f$ be a function from $G$ to $H$.
We say $f$ is a homomorphism from $\Gamma(G)$ to $\Gamma(H)$ whenever 
$f$ maps endpoints of each edge in $\Gamma(G)$ to endpoints of an edge in $\Gamma(H)$. Formally, $\{u,v\}\in Ed(G)$ implies $\{f(u),f(v)\}\in Ed(H)$, for all pairs of vertices $u, v \in V(G)$. 
Clearly, if $f$ is a bijection, we say $\Gamma(G)$ is isomorphic to
$\Gamma(H)$, and we denote this by $\Gamma(G)\cong \Gamma(H)$.

\begin{thm}
Let  $G$ be a finite  group of  order $n$. Then $\Gamma(G)\cong \Gamma(C_n)$ if and only if $G$ is nilpotent an  every Sylow subgroup of $G$ is in $CP2$.
\end{thm}
\begin{proof}
{If $G$ is nilpotent an  every Sylow subgroup of $G$ is in $CP2$, then by Lemma \ref{1}, $LC(G)=G$. So $\Gamma(G)$ is a complete graph and so
is isomorphic to $\Gamma(C_n)$.

Suppose that $\Gamma(G)\cong \Gamma(C_n)$. First suppose that $G$ is a $p$-group. Let $x\in G$.  Since $\Gamma(G)$ is a complete graph for all $y\in G$,  
$x$ is connected to $y$. Then $o(xy)\mid lcm(o(x),o(y))=max\{o(x),o(y)\}$.
Therefore  $G\in CP2$.  So suppose that $G$ is not a $p$-group.
Let $x\in G\setminus E(G)$ and $n$ an integer.
Since $\Gamma(G)$ is a complete graph for all $y\in G$,  
$x^n$ is connected to $y$. Then $o(x^ny)\mid lcm(o(x^n),o(y))$, and hence
$x\in LCM(G)$.  Since $G$ is not a $p$-group, we have
$LC(G)=G$. By Lemma \ref{1},
$G$ is nilpotent and all Sylow subgroups of $G$ are in $CP2$.
}
\end{proof}

In the following lemma we find a minimal and a maximal bound for $Deg(G)$. 
\begin{lem}
Let $G$ be a finite group of order $n$.
Then $$|G|(h(G)+1)\leq Deg(G)\leq n(n+1).$$
\end{lem}

\begin{proof}
{We have $h(G)|G|=\sum_{g\in G}|C_G(g)|$. Let $g\in G$. We denote the set of all vertices in $G$ which they are adjacent of $g$ in $\Gamma(G)$ by $Ed_G(g)$. We know that for all $x\in N_G(\langle g\rangle)$, we have  $o(xg)\mid lcm(o(x),o(g))$,  so

\begin{eqnarray*}
Deg(G)&=& \sum_{g\in G}(|Ed_G(g)|+1)\\&\geq& \sum_{g\in G}(|N_G(\langle g\rangle)|+1)\\&\geq& \sum_{g\in G}(|C_G(g)|+1)\\&=&|G|h(G)+|G|\\&=&|G|(h(G)+1).
\end{eqnarray*}

}
\end{proof}
Clearly, when $G$ is a abelian group, then $\Gamma(G)$ is a complete graph.
In this case $Deg(G)=|G|(|G|+1)=|G|(h(G)+1)$.
A natural  question is the following:
Is there exist  a non-abelian group $G$ such that $Deg(G)=|G|(|G|+1)=|G|(h(G)+1)$?

The answer to this question is no.
\begin{thm}
Let $G$ be a finite group. Then  $Deg(G)=|G|(h(G)+1)$ if and only if  $G$ is an abelian group.
\end{thm}
\begin{proof}{ If $G$ is an abelian group, then clearly,  $Deg(G)=|G|(|G|+1)=|G|(h(G)+1)$. So suppose that $Deg(G)=|G|(h(G)+1)$. We prove that
$G$ is an abelian group. We proceed by induction on $|G|$. 
The base of induction is trivial. Suppose for a contradiction that $G$ is not abelian group.
Let $H$ be a maximal   subgroup of $G$.
If there is $a\in H$ such that 
$Ed_H(a)\neq C_H(a)$, then  $Ed_G(a)\neq C_G(a)$, and so
\begin{eqnarray*}
Deg(G)&=& \sum_{g\in G}(|Ed_G(g)|+1)\\&\geq& \sum_{g\in G}(|N_G(\langle g\rangle)|+1)\\&>& \sum_{g\in G}(|C_G(g)|+1) \\&=&|G|(h(G)+1),
\end{eqnarray*}
 which is a contradiction.
So for any $a\in H$ we have 
$|Ed_G(a)|=|N_H(\langle a\rangle)|=| C_H(a)|$.
By induction hypothesis, 
$H$ is  abelian.  Since all proper subgroups of $G$ are abelian, $G$ is a minimal non-abelian group.  First suppose that $G$ is a $p$-group. If $G$ is a regular group, then $G\in CP2$, and so 
$\Gamma(G)$ is a complete graph. It follows that
$|Ed_G(g)|=|G|=|C_G(g)|$ for all $g\in G$, and so $G$ is an abelian group, which is a contradiction.
So suppose that $G$ is not a regular $p$-group.
Since all maximal subgroups of $G$ are abelian,  $G$ is minimal irregular group. Then by Theorem 3 of  \cite{man} $G=\langle a, b\rangle$, $exp(G')=p$ and  $Z(G)$ is cyclic of order $exp(G)/p$.
Let $M$ and $H$ be two distinct maximal subgroups of $G$.
Clearly,  $M\cap H\leq Z(G)$.  Then $G'\leq Z(G)$, and so
$G$ is a nilpotent group of class two. If $p>2$, then   $G$ is a regular group, which is a contradiction. 
So   $p=2$. 
If $|Z(G)|\leq 4$, then $|G|\leq 16$. A GAP computation shows that 
 $G$ is an abelian group. So suppose that 
 $|Z(G)|>4$. Then $exp(G)>4$. It follows that $G^4=\langle\{ x^4: x\in G\}\rangle\neq 1$. Since $G'\leq G^4\leq Z(G)$, we have $G$ is a powerful group. Since all powerful groups are in $CP2$, we have 
 $Deg(G)=|G|(|G|+1)=|G|(h(G)+1)$, which means that $G$ is abelian group, which is a contradiction. 
So suppose that $G$ is not a $p$-group. If $G$ is a nilpotent group, then 
all Sylow subgroups of $G$ are abelian, and so $G$ is abelian, which is a contradiction. 
So suppose that $G$ is not a nilpotent group.
Since $G$ is a minimal non-abelian group, by Theorem \ref{11}, $G=P\rtimes Q$ where $P\in Syl_p(G)$ and $Q=\langle a\rangle\in Syl_q(G)$ is 
a cyclic subgroup. Let $g\in G\setminus a^{-1}(P\setminus\{1\})$. Then   $o(ag)\nmid lcm(o(g),o(a))$, so  $|Ed_G(a)|+1=deg(a)=|G|-(|P|-1)+1>|N_G(\langle a\rangle)$.
 It follows that
\begin{eqnarray*}
Deg(G)&>& \sum_{g\in G}(|N_G(g)|+1)\\&\geq& \sum_{g\in G}(|C_G(g)|+1)\\&=&|G|(h(G)+1),
\end{eqnarray*}

which is a contradiction.}

\end{proof}

Let $G$ and $H$ be two finite groups. The following lemma shows that  $Deg(H\times G)\neq Deg(H)Deg(G)$, and so $Deg(\cdot)$ is not a multiplicative function, but    the valor of $Deg(H\times G)$ is bounded in terms of $Deg(H)$, $Deg(G)$, $|H|$ and $|G|$. 
\begin{lem}\label{ineq}
(a) Let $G$ and $H$ be two finite groups.
Then $$Deg(H\times G)\geq (Deg(G)-|G|)(Deg(H)-|H|)+|H||G|.$$
Also, if  $gcd(|H|,|G|)=1$, then 
$$Deg(H\times G)=(Deg(G)-|G|)(Deg(H)-|H|)+|H||G|.$$

(b) Let $L=G\rtimes H$ be a finite group where $G\in Syl_p(L)$ such that $G\in CP2$ and $H\leq L$. Then  $$Deg(H\times G)\geq Deg(L).$$
\end{lem}
\begin{proof}
{
Let $g$ and $y$ be adjacent in $\Gamma(G)$ and   $a$ and $b$ be adjacent in $\Gamma(H)$. Then $o(gy)\mid lcm(o(g),o(y))$ and 
$o(ab)\mid lcm(o(a),o(b))$. We have 
\begin{eqnarray*}
o((a,g)(b,y))&=& lcm(o(ab),o(gy))\\&\mid& lcm(lcm(o(a),o(b)),lcm(o(g),o(y)))\\&=&lcm(o((a,g)),o((b,y)))
\end{eqnarray*}
That means, $(a,g)$ and $(b,y)$ are connected.
Let $\{y_1,\ldots ,y_k\}$ and $\{b_1,\ldots ,b_t\}$ be the set of all elements of $G$ and $H$ which are connected to 
$g$ and $a$, respectively.
Clearly, $deg(g)=k+1$ and $deg(a)=t+1$.
It follows that $deg((a,g))\geq ((deg_H(a)-1)\cdot (deg_G(g)-1))+1$.
Hence
\begin{eqnarray*}
Deg(H\times G)&=&\sum_{(a,g)\in H\times G}deg((a,g))\\&\geq&
\sum_{(a,g)\in H\times G}((deg_H(a)-1)\cdot (deg_H(g)-1))+1\\&=&
Deg(H)Deg(G)-|G|Deg(H)-|H|Deg(G)+2|H||G|.
\end{eqnarray*}

If $gcd(|G|,|H|)=1$.
  Let $(b,y)\in H\times G$ such that
$(a,g)$ and $(b,y)$ are connected.
Then    $o((a,g)(b,y))=o((ab,gy))\mid lcm(o((a,g)),o((b,y)))$.
Since $gcd(|H|,|G|)=1$, we have  $o(gy)\mid lcm( o(g),o(y))$ and $o(ab)\mid lcm( o(a),o(b))$. Consequently,
 $g$ and $y$ are connected. Also, $a$ and $b$ are connected.
 So $deg(a,g)\leq (deg_H(a)-1)\cdot (deg_G(g)-1)+1$. It follows that 
 \begin{eqnarray*}
Deg(H\times G)&=&\sum_{(a,g)\in H\times G}deg((a,g))\\&\leq&
\sum_{(a,g)\in H\times G}(deg_H(a)-1)\cdot (deg_G(g)-1)+1\\&=&
Deg(H)Deg(G)-|G|Deg(H)-|H|Deg(G)+2|H||G|.
\end{eqnarray*}
 
 Consequently,  $Deg(H\times G)=Deg(H)Deg(G)-|G|Deg(H)-|H|Deg(G)+2|H||G|$.

(b) Let $g\in G$ and $h\in H$.  We have $(hg)^{o(h)}=g^{h^{-1}}g^{h^{-2}}\ldots g^{h^{-o(h)}}$. Since $G\in CP2$, we have 
$(g^{h^{-1}}g^{h^{-2}}\ldots g^{h^{-o(h)}})^{o(g)}=1$, and so
$o(gh)\mid o(g)o(h)=lcm(o(g),o(h))$.

Let $g$ and $y$ be adjacent in $\Gamma(G)$ and   $a$ and $b$ be adjacent in $\Gamma(H)$. Then $o(gy)\mid lcm(o(g),o(y))$ and 
$o(ab)\mid lcm(o(a),o(b))$. We have 
\begin{eqnarray*}
o((a,g)(b,y))&=& lcm(o(ab),o(gy))\\&\mid& lcm(lcm(o(a),o(b)),lcm(o(g),o(y)))\\&=&lcm(o((a,g)),o((b,y)))
\end{eqnarray*}
That means, $(a,g)$ and $(b,y)$ are connected.
Let $\{y_1,\ldots ,y_k\}$ and $\{b_1,\ldots ,b_t\}$ be the set of all elements of $G$ and $H$ which are connected to 
$g$ and $a$, respectively.
Clearly, $deg(g)=k+1$ and $deg(a)=t+1$.
It follows that $deg((a,g))\geq ((deg_H(a)-1)\cdot (deg_G(g)-1))+1$.
Hence
\begin{eqnarray*}
Deg(H\times G)&=&\sum_{(a,g)\in H\times G}deg((a,g))\\&\geq&
\sum_{(a,g)\in H\times G}((deg_H(a)-1)\cdot (deg_H(g)-1))+1\\&=&
Deg(H)Deg(G)-|G|Deg(H)-|H|Deg(G)+2|H||G|.
\end{eqnarray*}

}
\end{proof}

In general,  if $Deg(H\times G)=(Deg(G)-|G|)(Deg(H)-|H|)+|H||G|$, then   
the equality $gcd(|H|,|G|)=1$ is not hold.
For example, Let $H=G=C_2$.
Then $Deg(C_2\times C_2)=4\cdot 5=20$.
Also, $(Deg(G)-|G|)(Deg(H)-|H|)+|H||G|=16+4=20$.
\begin{cor}\label{ineq2}
Let $H$ and $G$ be two  finite groups of order $n$,and let $A$ be a finite group such that $gcd(|A|,n)=1$. Then 
$Deg(A\times H)\leq Deg(A\times G)$ if and only if $Deg(H)\leq Deg(G)$.
\end{cor}
\begin{proof}
{ By Lemma \ref{ineq}, we have 
$Deg(A\times G)=(Deg(G)-n)(Deg(A)-|A|)+n \cdot |A|$ and
$Deg(A\times H)=(Deg(H)-n)(Deg(A)-|A|)+n \cdot |A|$.
 Then 
$Deg(A\times H)\leq Deg(A\times G)$ if and only if $Deg(H)\leq Deg(G)$. 

}
\end{proof}
Let $Gr_2(n)$ be the set of all groups $G\in Gr(n)$ in which  $Deg(G)< n(n+1)$.
The other interesting question is the following: Let $G\in Gr_2(n)$ such that $Deg(H)\leq Deg(G)$ for all $H\in Gr_2(n)$. What can we say about the structure of $G$?

We answer this question in the case $n$ is a square free integer. 

The following  lemma is useful to prove Lemma \ref{square}.
\begin{lem}\label{in1}
Let $0\leq a\leq b\leq n$ be   integer numbers.
Then \begin{equation}\label{equ22}
a(n+1)+(n-a)(n-a+1)\leq b(n+1)+(n-b)(n-b+1)
\end{equation}

\end{lem}
\begin{proof}
{We proceed by induction on $b$. Induction start trivially whenever
$b=a$. The equation \ref{equ22}, is equal to
$$(b-a)(n+1)+ (n-b)(n-b+1)-(n-a)(n-a+1)\geq 0$$
which is equal to 
$$(b-1-a)(n+1)+ (n-(b-1))(n-b+1)-(n-a)(n-a+1)+(n+1-(n-b+1))\geq 0.$$
By induction hypothesis, we have 
$$(b-1-a)(n+1)+ (n-(b-1))(n-b+1)-(n-a)(n-a+1)\geq 0.$$
Since $n+1-(n-b+1)\geq 0$, we have 
$$(b-a)(n+1)+ (n-b)(n-b+1)-(n-a)(n-a+1)\geq 0.$$
}
\end{proof}
\begin{lem}\label{square}
Let $G$ and $H$ be two groups of order $n$ where $n=p_1p_2\ldots p_k$ where $p_1<\ldots <p_k$ are primes.

(a) Suppose that $G$ is a Frobenius group.  If  $|Fit(G)|\leq |Fit(H)|$, then    $Deg(H)\geq Deg(G)$.

(b) If $p_1\mid p_2-1$ and $G$ is not abelian, then $Deg(G)\leq Deg(C_{n/p_2p_1}\times (C_{p_2}\rtimes C_{p_1}))$.
\end{lem}

\begin{proof}
{Clearly, we may assume that $G\neq Fit(G)$. We have $G=P_k\rtimes (E\times \langle a\rangle)$ and  $H=Q_k\rtimes (F\times \langle b\rangle)$ where 
$|P_k|=p_k=|Q_k|$. Let $Fit(G)=\langle x\rangle$ and $Fit(H)=\langle y\rangle$. Let $z\in Fit(G)$. For all $t\in G$, we have 
$o(zt)\mid lcm(o(z),o(t))$, so $Fit(G)=LC(G)$.
Let $u\in \langle a\rangle\setminus 1$.
If $t\not\in u^{-1}(Fit(G)\setminus \{1\})$, then 
$o(ut)\mid lcm(o(u),o(t))$, and so 
$deg(u)=|G|-|Fit(G)|$. It follows that $$Deg(G)=|Fit(G)|(|G|+1)+(|G|-|Fit(G)|)(|G|-|Fit(G)|).$$ 
Let $z\in Fit(H)$. For all $t\in H$, we have 
$o(zt)\mid lcm(o(z),o(t))$, so $Fit(H)=LC(H)$.
Let $u\in \langle b\rangle\setminus 1$.
If $t\not\in u^{-1}(Fit(H)\setminus C_{Fit(H)}(u)$, then 
$o(ut)\mid lcm(o(u),o(t))$, and so 
$$deg(u)=|H|-(|Fit(H)|\setminus C_{Fit(H)}(u))+1\geq |H|-|Fit(H)|.$$
It follows that 
$$Deg(H)\geq |Fit(H)|(|G|+1)+(|G|-|Fit(H)|)(|G|-|Fit(H)|).$$

By Lemma \ref{in1},  $Deg(G)\leq Deg(H)$.

(b) Let $S=C_{n/p_2p_1}\times (C_{p_2}\rtimes C_{p_1})$. We have $G=Fit(G)\rtimes \langle a\rangle.$ 
Let $u\in \langle a\rangle\setminus 1$.
If $t\not\in u^{-1}(Fit(G)\setminus C_{Fit(H)}(u))$, then 
$o(ut)\mid lcm(o(u),o(t))$, and so 
$deg(u)=|G|-(|Fit(G)|-C_{Fit(H)}(u))+1\leq |G|-\frac{n}{p_1}+1$.
Clearly, if $y\in S\setminus Fit(S)$, then 
$deg(y)=|G|-\frac{n}{p_1}+1$.
Since $|G\setminus Fit(G)|\geq |H\setminus Fit(H)|$, we have
$Deg(G)\leq Deg(H).$ 

}
\end{proof}
\begin{thm}\label{sq2}
Let $G$ be a group of order  $n=p_1p_2\ldots p_k$ where $p_1<\ldots <p_k$ are primes. 
Let $p_i$ be the smallest prime divisor of $n$ such that $p_i\mid p_j-1$ for some prime divisor $p_j$ of $n$.
Then  $$Deg(G)\leq Deg(C_{n/p_rp_i}\times (C_{p_r}\rtimes C_{p_i}))$$
where $p_r=min\{p_j\mid n:  p_i\mid p_j-1\}$.
\end{thm}
\begin{proof}
{We have $G=T\times F$ where $T\leq Z(G)$ and for all
prime divisors  $p$ of $n$ which $p<p_i$, we have $p\mid |T|$.
 By Corollary  \ref{ineq2}, we may assume that $T=1$. Hence the proof is clear by Lemma \ref{square} (b).

}
\end{proof}
Let   $n=p_1^{\alpha_1}\ldots p_k^{\alpha_k}$ where $p_1<\ldots <p_k$ are primes   such that $k>1$ and $p_1\mid p_2-1$. Let $X$ be the set of all minimal non-nilpotent groups of order $d$ where $d\mid n.$ We denote by $A(n)$
to be a group in $X$ of minimal order i. e.  $|A(n)|\leq |T|$ for all $T\in X$. 
In view of Theorem \ref{sq2}, we believe  that the following problem is true. 
\begin{que}
(a) Let   $n=p_1^{\alpha_1}\ldots p_k^{\alpha_k}$ where $p_1<\ldots <p_k$ are primes   such that $k>1$ and $p_1\mid p_2-1$. 
Let $G=C_{n/|A(n)|}\times A$.
Let $H$ be a non-nilpotent group of order 
$n$. Is it true that $Deg(H)\leq Deg(G)$?

(b) Let   $n=p_1^{\alpha_1}\ldots p_k^{\alpha_k}$ where $p_1<\ldots <p_k$ are primes  such that $\alpha_1\geq p^{p+1}$. 
Let $G=C_{n/p_2p_1}\times (C_{p_1}\wr C_{p_1})$.
Let $H\in Gr_2(n)$ be a nilpotent group. Is it true that $Deg(H)\leq Deg(G)$?

\end{que}
For a finite group $G$ and natural number $n$, set $ G(n)=\Omega_n(G)$  and define the type of $G$ to be the function whose value at $n$ is the
order of $G(n)$. Is it true that a group is soluble if its type is the same as that of a
soluble one? 
This is  the unsolved Problem 12.37 of the \cite{kok}.
The part (a) of the following question is similar to  this question with respect to the $LCM$-graph of a finite group.
\begin{que}\label{pro} Let $G,H\in Gr(n)$.

(a)  Suppose that  $G$ is a solvable group.  
  Is it true that 
if $\Gamma(G)\cong \Gamma(H)$, then $H$ is a solvable group?

(b) Suppose that  $G$ is a simple group. 
 Is it true that 
if $\Gamma(G)\cong \Gamma(H)$, then $G\cong H$?

\end{que}

Let  $G,S\in Gr(n)$ such that $S$ is a simple group. Is it   true  that  $Deg(S)\leq Deg(G)$? The answer is no.
For example, let $S=A_6$ and  let $G=(C_{18} \times D_{10}) \rtimes  C_2$ which is obtained by the following order in  GAP computation:

g:=AllSmallGroups(360,IsNilpotent,false);;

G:=g[11];

Then $Deg(S)=70560>Deg(G)=64800.$ 
So, if $Deg(H)$ is  minimum in the set of   all groups  of the same order $n$, then  $H$ could be a solvable group. 
Also, there is a solvable group $G$ of order 180 such that
$Deg(G)=Deg(GL(2,4))$.

 Another interesting   question is the following:
 \begin{que}Let $\Delta$ be a graph with $n$ vertices. Let $G(\Delta)=\{H\in Gr(n): \Gamma(H)\cong \Delta\}$. What can we say about  the structure of group $H$ in $G(\Delta)$ with respect to $\Delta$?

\end{que}

For example, let $\Delta$ be a graph with $4$ vertices such that $\Delta$ is not a complete graph. Then there is no any group of order $4$ such that whose graph is isomorphic to $\Delta$.  Because, all groups of order $4$ are abelian, and so there graph is complete.

A regular graph is a graph where all vertices  have the same number of degree.   A regular graph with vertices of degree $k$ is called a $k$-regular graph or regular graph of degree $k$  

\begin{que}
Let $G$ be a group of  order $n$, and let $\Delta=\Gamma(G)\setminus \Gamma(LC(G))$. Is it true that if $\Delta$ is a $k$-regular group, then $G$ is a solvable group?

\end{que}

\bigskip
{\it Author's Adresses:}

\medskip
Mohsen Amiri\\
Universidade Federal do Amazonas,\\
Departamento de Matem\'atica - ICE-UFAM\\  
69080-900, Manaus-AM\\
Brazil\\

\medskip
Igor Lima,\\
 \\ Universidade de Bras\'ilia,\\
70910-900 Bras\'ilia - DF\\
Brazil

\end{document}